\documentclass[11pt]{amsart}

\usepackage{enumerate}
\usepackage{amsmath,amsthm,verbatim,amssymb,amsfonts,amscd, graphicx}
\usepackage{tikz-cd} 
\usepackage{graphics}
\usepackage{hyperref}
\usepackage{mathrsfs}
\usepackage{relsize}
\usepackage{tikz-cd}
\usepackage{amsthm}
\usepackage{thmtools}
\usepackage[sort cites=true, backend=biber]{biblatex}
\theoremstyle{definition}

\usepackage{biblatex}

\begin{document}

\title{Generalization of Abhyankar's Lemma to henselian valued fields}
\author{Arpan Dutta}
\def\NZQ{\mathbb}               
\def\NN{{\NZQ N}}
\def\QQ{{\NZQ Q}}
\def\ZZ{{\NZQ Z}}
\def\RR{{\NZQ R}}
\def\CC{{\NZQ C}}
\def\AA{{\NZQ A}}
\def\BB{{\NZQ B}}
\def\PP{{\NZQ P}}
\def\FF{{\NZQ F}}
\def\GG{{\NZQ G}}
\def\HH{{\NZQ H}}
\def\UU{{\NZQ U}}
\def\P{\mathcal P}

%
%
\let\union=\cup
\let\sect=\cap
\let\dirsum=\oplus
\let\tensor=\otimes
\let\iso=\cong
\let\Union=\bigcup
\let\Sect=\bigcap
\let\Dirsum=\bigoplus
\let\Tensor=\bigotimes

\newtheorem{Theorem}{Theorem}[section]
\newtheorem{Lemma}[Theorem]{Lemma}
\newtheorem{Corollary}[Theorem]{Corollary}
\newtheorem{Proposition}[Theorem]{Proposition}
\newtheorem{Remark}[Theorem]{Remark}
\newtheorem{Remarks}[Theorem]{Remarks}
\newtheorem{Example}[Theorem]{Example}
\newtheorem{Examples}[Theorem]{Examples}
\newtheorem{Definition}[Theorem]{Definition}
\newtheorem{Problem}[Theorem]{}
\newtheorem{Conjecture}[Theorem]{Conjecture}
\newtheorem{Question}[Theorem]{Question}
\let\epsilon\varepsilon
\let\phi=\varphi
\let\kappa=\varkappa

\def \a {\alpha}
\def \s {\sigma}
\def \d {\delta}
\def \g {\gamma}

%
%
\textwidth=15cm \textheight=22cm \topmargin=0.5cm
\oddsidemargin=0.5cm \evensidemargin=0.5cm \pagestyle{plain}

\address{Department of Mathematics, IISER Mohali,
	Knowledge City, Sector 81, Manauli PO,
SAS Nagar, Punjab, India, 140306.}
\email{arpan.cmi@gmail.com}


\begin{abstract}
	Abhyankar showed that for a finite tame extension $L_1/K$ and a finite extension $L_2/K$ of $\mathfrak{P}$-adic fields, the condition $[\nu L_1 : \nu K]$ divides $[\nu L_2 : \nu K]$ is sufficient to eliminate ramification, that is, $L_1 \cdot L_2 / L_2$ is unramified. In this paper, we show that the above condition is not sufficient in the case of an arbitrary henselian valued field. We construct a counterexample illustrating that fact. We also give a necessary and sufficient condition for the elimination of tame ramification of a henselian field after a finite extension of the base field.
\end{abstract}

\maketitle

\section*{Notations} For any two fields $K_1$ and $K_2$, their compositum will be denoted by $K_1 \cdot K_2$. The algebraic closure of a field $K$ will be denoted by $\overline{K}$ and the characteristic will be denoted by char $K$. If char $K = p > 0$, its perfect hull will be denoted by $K^{\frac{1}{p^\infty}}$. The degree of a finite extension $L/K$ of fields will be denoted by $[L:K]$. The order of a finite quotient $G/H$ of groups will be denoted by $[G:H]$. 

\section{Introduction}

A $\mathfrak{P}$-adic field is the completion of an algebraic number field under a discrete valuation of rank 1. Abhyankar's lemma [cf. \ref{Narkiewicz}, Chapter 5, Corollary 4] gives a necessary and sufficient condition to eliminate tame ramification of a $\mathfrak{P}$-adic field after lifting through a finite extension. 
\begin{Lemma}\textbf{(Abhyankar)}\label{Abhyankar lemma}
	Let $K$ be a $\mathfrak{P}$-adic field, $L_1/K$ be a finite tame extension and $L_2/K$ be a finite extension. Let $e_1 = [\nu L_1 : \nu K]$ and $e_2 = [\nu L_2 : \nu K]$ denote the corresponding ramification indices. Let $L = L_1 \cdot L_2$. Then $L/L_2$ is unramified if $e_1$ divides $e_2$.
\end{Lemma}
 For a valued field of rank 1, the henselization lies in the completion, and in particular, the completion itself is henselian. So the question naturally arises whether the above condition holds for an arbitrary henselian valued field. In section 3, we construct the following counter-example (Example \ref{Counterexample to abhyankar lemma}) to show that this is not the case. 
 \begin{Example} 
 	Let $k$ be a field with char $k \neq 2$ and $k(X,Y)$ denote the function field in two variables. Define a map $\nu : k[X,Y] \longrightarrow \ZZ \dirsum \ZZ$ by 
 	\begin{equation*}
 	\nu (\sum_{a,b \in \NN} c_{a,b} X^a Y^b) = \text{min} \{(a,b) \mid c_{a,b} \neq 0 \}
 	\end{equation*}
 	where the ordering taken is lex order. We extend $\nu$ to a valuation on $k(X,Y)$ by defining $\nu (\frac{f}{g}) := \nu(f) - \nu(g)$ where $f,g \in k[X,Y]$. Set $K = (k(X,Y), \nu)^h$ to be the henselization of $k(X,Y)$ with respect to $\nu$. Let $L_1 = K(\sqrt{X})$ and $L_2 = K(\sqrt{Y})$. Let $L = L_1 \cdot L_2$. Then $(L/L_2, \nu)$ is ramified.  
 \end{Example}
 
 A local field is a complete discretely valued field of rank 1, with finite residue field. In [\ref{CH}] it is shown that the ramification index of the compositum of two finite extensions of local fields is equal to the least common multiple of the ramification indices corresponding to the finite extensions, provided at least one of the extensions is tame. Example \ref{Counterexample to abhyankar lemma} shows that this result also does not hold in the case of arbitrary henselian valued fields. 
\par $\mathfrak{P}$-adic fields are henselian and in particular, defectless. In section 4, we see that the defect does not play a role in the elimination of tame ramification, since it is enough to lift through a finite tame extension. We state two important results. For a valued field $(K, \nu)$, $K^r$ denotes the absolute ramification field [cf. Section 2]. 
\begin{Lemma}
	Let $(K, \nu)$ be a henselian field such that $K \nu$ is perfect. Let $(L_1/K, \nu)$ be a finite tame extension and $(L_2/K, \nu)$ be a finite extension. Let $L = L_1 \cdot L_2$, $L_2 ^{\prime} = L_2 \sect K^r$ and $L^{\prime} = L_1 \cdot L_2 ^{\prime}$. Then $(L/L_2, \nu)$ and $(L^{\prime}/ L_2 ^{\prime} , \nu)$ are finite tame extensions. Moreover $(L/L_2, \nu)$ is unramified if and only if $(L^{\prime} / L_2 ^{\prime}, \nu)$ is unramified. 
\end{Lemma}

\begin{Theorem}
	Let $(K, \nu)$ be a henselian field such that $K \nu$ is perfect. Let $(L_1/K, \nu)$ be a finite tame extension and $(L_2/K, \nu)$ be a finite extension. Set $L = L_1 \cdot L_2$. Set $L_2 ^{\prime} = L_2 \sect K^r$ and $L^{\prime} = L_1 \cdot L_2 ^{\prime}$. Set $e_1 = [\nu L_1 : \nu K]$  and $e_2 ^{\prime} = [\nu L_2 ^{\prime} : \nu K]$. For any $\theta \in \mathfrak{O}_{L^{\prime}}$ let $d_{\theta} = [L_1 \nu (\theta \nu) : L_1 \nu]$. Then, 
	\begin{equation*}
	(L/L_2, \nu) \text{ is unramified if and only if } e_2 ^{\prime} d_{\theta} = e_1 [L^{\prime} : L_1] \text{ for some } \theta \in \mathfrak{O}_{L^{\prime}}.
	\end{equation*} 
	In this case, $L^{\prime} \nu = L_1 \nu (\theta \nu)$ and $e_1$ divides $e_2 ^{\prime}$.
\end{Theorem} 

Another important related problem is the problem of uniformization of valued field extensions. It has been shown by Knaf and Kuhlmann [\ref{Knaf Kuhlmann}, \ref{Kuh local uzn arbit char}] that any valuation of an algebraic function field admits local uniformization in a finite extension of the function field. It is also known that a finite unramified extension of henselian valued fields is uniformizable [cf. \ref{Kuh vln model}, Proposition 15.8]. In section 4, we construct an easy example of a family of valued field extensions which are uniformizable after lifting through a finite extension.

\section{Preliminaries}

A valuation $\nu$ on a field $K$ is a homomorphism $\nu : K^* \longrightarrow \Gamma$, where $\Gamma$ is an ordered abelian group, such that $\forall x,y \in K^*$,
\begin{equation*}
	\nu (xy) = \nu (x) + \nu(y), \, \nu(x+y) \geq \text{min }\{ \nu(x), \nu(y) \}. 
\end{equation*}  
We have the convention that $\nu(0) = \infty$ where $\infty$ has a higher value than any element in $\Gamma$. We can assume that $\nu$ is surjective onto $\Gamma$. $\Gamma$ is called the value group and will be denoted by $\nu K$. A valued field $(K, \nu)$ is a field $K$ with a valuation $\nu$ on it. The valuation ring $\mathfrak{O}_K = \{ x \in K \mid \nu(x) \geq 0 \}$ is a normal local domain with the unique maximal ideal $\mathfrak{m}_K = \{ x \in K \mid \nu(x) > 0  \}$. The quotient field of $\mathfrak{O}_K$ is $K$. We will denote the residue field $\frac{\mathfrak{O}_K}{\mathfrak{m}_K}$ as $K \nu$. For $x \in \mathfrak{O}_K$, its value will be denoted by $\nu (x)$ and its residue will be denoted by $x \nu$. For a polynomial $f(X) \in \mathfrak{O}_K [X]$, its image in $K \nu [X]$ under the natural surjection map will be denoted by $(f \nu)(X)$. 
\newline The \textbf{residue characteristic} of $(K, \nu)$ will be denoted by $p$ and is defined as 
\begin{equation*}
	 p = \left\{  
	\begin{aligned}
	& \text{char } K \nu \text{ if char }K \nu > 0 \\ & 1 \text{ if char }K \nu = 0
	\end{aligned}
	\right\}.
\end{equation*}
\newline Given a valued field $(K, \nu)$ and a field extension $L/K$, we always obtain an extension of $\nu$ to $L$ [\ref{ZS2}]. Conversely if $L/K$ is a field extension and $\nu$ is a valuation on $L$, its restriction to $K$ is a valuation on $K$. We will denote an extension of valued fields as $(L/K, \nu)$, which will mean $L/K$ is a field extension, $\nu$ is a valuation on $L$ and $K$ is equipped with the restricted valuation. In this case we have that $\nu K$ is a subgroup of $\nu L$ and $K \nu$ is a subfield of $L \nu$. An extension $(L/K, \nu)$ of valued fields is said to be \textbf{immediate} if $\nu K = \nu L$ and $K \nu = L \nu$. 
\newline If $L/K$ is a finite extension, there are finitely many extensions of $\nu$ from $K$ to $L$. Let they be denoted by $\{\nu_1, \cdots , \nu_g  \}$. Associated to these extensions, we have two classical indices. The \textbf{ramification index} is defined as $e_i = [\nu_i L : \nu K]$ and the \textbf{residual degree} is defined as $f_i = [L \nu_i : K \nu]$. We have $1 \leq e_i, f_i \leq [L : K]$. When we have an unique extension of $\nu$ to $L$, an important related result is \textbf{Ostrowski's Lemma}.
\begin{Lemma}\textbf{(Ostrowski)}\label{Ostrowski}
	Let $(K, \nu)$ be a valued field with residue characteristic $p$. Suppose that $L/K$ is a finite extension such that there is an unique extension of $\nu$ to $L$. Then, 
	\begin{equation*}
		[L : K] = d \cdot e \cdot f
	\end{equation*} 
	where $d$ is a power of $p$, $e = [\nu L : \nu K]$ and $f = [L \nu : K \nu]$. 
\end{Lemma} 
The number $d$ is called the \textbf{defect} of the extension $(L/K, \nu)$. If $d =1$ the extension is said to be defectless. In particular, if char $K \nu = 0$, the extension is always defectless. 
\par The \textbf{rank} of a valuation $(K, \nu)$ is rk $\nu$ = rk $\nu K$ = dim $\mathfrak{O}_K$ where dimension means the Krull dimension. The \textbf{rational rank} of $(K, \nu)$ is rat rk $\nu$ = rat rk $\nu K$ = dim$_{\QQ} (\nu K \tensor_\ZZ \QQ)$ where the dimension is as a $\QQ$ vector space. We have rk $\nu \leq$ rat rk $\nu$ [\ref{ZS2}]. Also the rank and rational rank of a valuation remain invariant under finite extensions of the valued field [\ref{ZS2}].
\par We now quickly recall some aspects of ramification theory [cf. \ref{Abhyankar ramification book}, \ref{Endler}, \ref{Kuh vln model}, \ref{ZS2}]. Let $(L/K, \nu) $ be an algebraic extension of valued fields such that the field extension $L/K$ is normal. Let $G = Aut (L/K)$. The set of all extensions of $\nu$ to $L$ are given by $\{ \nu \circ \sigma \mid \sigma \in G \}$. We define some important subgroups of $G$ and their corresponding fixed fields, and state some of their properties. By abuse of notation we will denote the restriction of $\nu$ to $K$ as $\nu$. The \textbf{decomposition group} of $\nu$ is defined as 
\begin{equation*}
	G^d = \{ \sigma \in G \mid \nu \circ \sigma = \nu \text{ on } L \}.
\end{equation*}
The \textbf{inertia group} is defined as 
\begin{equation*}
	G^i = \{ \sigma \in G \mid \nu (\sigma x - x ) > 0 \, \forall \, x \in \mathfrak{O}_L \}.
\end{equation*}
The \textbf{ramification group} is defined as 
\begin{equation*}
G^r = \{ \sigma \in G \mid \nu (\sigma x - x ) > \nu(x) \, \forall \, x \in \mathfrak{O}_L \}.
\end{equation*}
Let $S = K^{sep}$ denote the separable closure of $K$ in $L$. The corresponding fixed fields in $S$ will be denoted as $K^d$, $K^i$ and $K^r$ and are respectively called the \textbf{decomposition field}, \textbf{inertia field} and the \textbf{ramification field}. We have, 
\begin{equation*}
	G^r \trianglelefteq G^i \trianglelefteq G^d \leq G  \text{ and } G^r \trianglelefteq G^d.
\end{equation*}
Thus, Gal $(L/K^i) = G^i$, Gal $(L/K^r) = G^r$ and Gal $(K^r / K^i) = G^i / G^r$. We also have, 
\begin{itemize}
	\item $(K^d/ K, \nu)$ is an immediate extension, \\
	\item $\nu K^i = \nu K, \, K^i \nu = (K \nu)^{\text{sep}}$ is the separable closure of $K \nu \text{ in }L \nu$, \\
	\item $\nu K^r = (\nu L : \nu K)^{p^{\prime}} = \{ \gamma \in \nu L \mid n \gamma \in \nu K \text{ for some } n \in \NN \text{ coprime to }p \}, \, K^r \nu = K^i \nu$, \\
	\item $S/K^r \text{ is a } p-\text{extension}$.
\end{itemize}
In the case when $L = \overline{K}$, these are called the \textbf{absolute decomposition field (group), absolute inertia field (group)} and \textbf{absolute ramification field (group)}. 
 \par A valued field $(K, \nu)$ is said to be \textbf{henselian} if $\mathfrak{O}_K$ is a henselian ring, that is, it satisfies Hensel's Lemma [\ref{Nagata henselian}, \ref{Kuh vln model}, \ref{Kuh defect}]. From [\ref{Nagata henselian}, Lemma 3] and [\ref{ZS2}, Chapter VI, Section 7, Corollary 2], this is equivalent to saying that $\nu$ has an unique extension from $K$ to its algebraic closure $\overline{K}$, and hence to any algebraic extension of $K$. Thus Ostrowski's Lemma holds in the case of finite extensions of henselian valued fields. If $(K, \nu)$ is a henselian valued field, then $K = K^d$, the absolute decompostion field. An algebraic extension of a henselian valuation is again henselian [cf. \ref{Kuh vln model}, Corollary 1]. 
\newline For an arbitrary valued field $(K, \nu)$, $K^d$ is the least algebraic extension satisfying the Hensel's Lemma. We call $K^d$ to be the \textbf{henselization} of $K$. It will also be denoted by $K^h$. Thus henselization is always an immediate extension. 
\par Let $(L/K, \nu)$ be an algebraic extension of henselian valued fields. The extension $(L/ K, \nu)$ is said to be \textbf{tame} if for any finite subextension $L^{\prime}/K$, the following conditions are satisfied. 
\begin{itemize}
	\item $p$ does not divide $[\nu L^{\prime} : \nu K]$, \\
	\item $L^{\prime} \nu / K \nu \text{ is a separable extension}$, \\
	\item $(L^{\prime}/K, \nu) \text{ is a defectless extension}$.
\end{itemize}
It is known that for a henselian valued field $(K, \nu)$, the absolute ramification field $K^r$ is the unique maximal tame extension. Moreover, the absolute inertia field $K^i$ is the unique maximal tame extension such that there is no extension of value groups [cf. \ref{Kuh vln model}, Theorem 11.1]. In particular, a tame extension is always separable. 
\newline An extension $(L/K, \nu)$ of henselian valued fields is said to be \textbf{unramified} if $(L/K, \nu)$ is tame and $\nu L = \nu K$. From the preceding discussions, this is equivalent to the condition that $L \subseteq K^i$.
\newline The extension $(L/K, \nu)$ is said to be \textbf{purely wild} if $L/K$ and $K^r/K$ are linearly disjoint. Since $K^r/K$ is a Galois extension, this is equivalent to the condition $L \sect K^r = K$. Another equivalent characterization of a purely wild extension is, 
\begin{itemize}
	\item $\nu L / \nu K \text{ is a p-group}$,\\
	\item $L \nu / K \nu \text{ is purely inseparable}$.
\end{itemize}

\section{Abhyankar's Lemma}

In Lemma \ref{Abhyankar lemma} Abhyankar gave a sufficient condition to eliminate tame ramification after a finite extension of the base field [cf. \ref{Narkiewicz}, Chapter 5, Corollary 4]. That the condition is necessary follows easily from the multiplicative property of ramification indices. Evidently, $(L/L_2, \nu)$ is unramified implies $\nu L = \nu L_2$. Thus, with the notations as in Lemma \ref{Abhyankar lemma}, $e_2 = [\nu L : \nu K] = [\nu L : \nu L_1]e_1$ and hence $e_1$ divides $e_2$.
\par Since $\mathfrak{P}$-adic fields are defectless and in particular henselian, we want to investigate whether this condition extends to henselian valued fields. Let $(K, \nu)$ be a henselian valued field, $(L_1/K, \nu)$ be a finite tame extension and $(L_2/K, \nu)$ be an arbitrary finite extension. Let $e_i = [\nu L_i : \nu K], \, i = 1,2$. Let $L = L_1 \cdot L_2$ denote their compositum. Since an algebraic extension of a henselian field is henselian, we have that $L, L_1$ and $L_2$ are henselian. Suppose that $(L/ L_2, \nu)$ is unramified. Then using the same argument as above, we have that $e_1$ divides $e_2$ is a necessary condition to eliminate ramification. However the condition is not sufficient for henselian valued fields, as illustrated by the following example. 

\begin{Example} \label{Counterexample to abhyankar lemma}
		Let $k$ be a field with char $k \neq 2$ and $k(X,Y)$ denote the function field in two variables. Define a map $\nu : k[X,Y] \longrightarrow \ZZ \dirsum \ZZ$ by 
		\begin{equation*}
		\nu (\sum_{a,b \in \NN} c_{a,b} X^a Y^b) = \text{min} \{(a,b) \mid c_{a,b} \neq 0 \}
		\end{equation*}
		where the ordering taken is lex order. We extend $\nu$ to a valuation on $k(X,Y)$ by defining $\nu (\frac{f}{g}) := \nu(f) - \nu(g)$ where $f,g \in k[X,Y]$. Set $K = (k(X,Y), \nu)^h$ to be the henselization of $k(X,Y)$ with respect to $\nu$. Let $L_1 = K(\sqrt{X})$ and $L_2 = K(\sqrt{Y})$. Let $L = L_1 \cdot L_2$. Then $(L/L_2, \nu)$ is ramified.  
\end{Example}

\begin{proof}
	From [\ref{Abhyankar ramification book}, Chapter II, section 9, example 4], we have that $\nu$ is a valuation. The value group $\nu k(X,Y) = (\ZZ \dirsum \ZZ)_{lex}$. Since henselization is immediate, we have $\nu K = (\ZZ \dirsum \ZZ)_{lex}$. We further observe that $\nu$ is trivial on $k$. Thus $\nu k = 0$ and $k \nu = k \subseteq K \nu$. In particular, char $K \nu \neq 2$.  
	Since $K$ is henselian, we extend $\nu$ uniquely to $\overline{K}$. We know that $\nu \overline{K}$ is the divisible hull of $\nu K$, that is, $\nu \overline{K} = \QQ \tensor_\ZZ \nu K$. So, $\nu \overline{K}$ is $\QQ \dirsum \QQ$ equipped with the lex order. 
	\newline We observe that $2 \nu(\sqrt{X}) = \nu (X) = (1,0) \Longrightarrow \nu (\sqrt{X}) = (\frac{1}{2}, 0) \notin  \nu K$. Thus $\sqrt{X} \notin K$. Hence $[L_1 :K] = [K(\sqrt{X}) : K] = 2$. Again $Y \sqrt{X} \in L_1 \Longrightarrow (\frac{1}{2}, 1)  \in \nu L_1 \Longrightarrow \frac{1}{2} \ZZ \dirsum \ZZ \subset \nu L_1$. Thus $e_1 = [\nu L_1 : \nu K] \geq 2$. From Lemma \ref{Ostrowski} we then have, 
	\begin{equation*}
		[L_1 : K] = e_1 = 2, \, L_1 \nu = K \nu, \, \nu L_1 = \frac{1}{2}\ZZ \dirsum \ZZ .
	\end{equation*}
	Similarly, denoting $e_2 = [\nu L_2 : \nu K]$ we obtain,
	\begin{equation*}
		[L_2 : K] = e_2 = 2, \, L_2 \nu = K \nu, \, \nu L_2 = \ZZ \dirsum \frac{1}{2}\ZZ .
	\end{equation*}
	Now, $L = L_1 \cdot L_2 = K(\sqrt{X}, \sqrt{Y}) \Longrightarrow [L : K] \leq 4$. Again $\sqrt{X^m Y^n} \in L \, \forall \, m,n \in \ZZ$. Thus $\nu(\sqrt{X^m Y^n}) = (\frac{m}{2}, \frac{n}{2}) \in \nu L \, \forall \, m,n \in \ZZ $. Hence $\frac{1}{2} \ZZ \dirsum \frac{1}{2} \ZZ \subset \nu L$. Using the same arguments as above, we obtain 
	\begin{equation*}
		[L:K] = [\nu L : \nu K] = 4, \, L \nu = K \nu, \, \nu L = \frac{1}{2} \ZZ \dirsum \frac{1}{2} \ZZ .
	\end{equation*}
	Since char $K \nu \neq 2$, we observe that $L/K$ is a tame extension, and thus every subextension is also tame. Then $\nu L \neq \nu L_2$ implies that $(L/L_2, \nu)$ is a ramified extension. 
\end{proof}

\section{Elimination of ramification}
In this section, we give a necessary and sufficient condition to eliminate tame ramification of a henselian valued field after lifting through a finite extension. Throughout this section, $K^d, K^i$ and $K^r$ will respectively denote the absolute decomposition field, absolute inertia field and the absolute ramification field of a valued field $(K, \nu)$.

\begin{Lemma}\label{tame wild lemma}
	Let $(K, \nu)$ be a henselian field such that $K \nu$ is perfect. Let $(L_1/ K, \nu)$ be a finite tame extension and $(L_2/ K, \nu)$ be a finite purely wild extension. Set $L = L_1 \cdot L_2$. Then $[\nu L : \nu K] = [\nu L_1 : \nu K] [\nu L_2 : \nu K]$. Moreover $(L/L_1, \nu)$ is purely wild and $(L/L_2, \nu)$ is tame. 
\end{Lemma}

\begin{proof} 
Let $p$ denote the residue characteristic of $(K, \nu)$. Set $e = [\nu L : \nu K], \, f = [L \nu : K \nu], \,  e_i = [\nu L_i: \nu K], \, q_i = [\nu L : \nu L_i], \, f_i = [L_i \nu : K \nu], \, g_i = [L \nu : L_i \nu], \, i = 1,2$. The defect of any extension $(F_1 / F_2, \nu)$ will be denoted by $d(F_1/ F_2, \nu)$. If the valuation $\nu$ is tacitly understood, we will simply use the notation $d(F_1/ F_2)$.
\newline Since the ramification indices and residual degrees are multiplicative, we have
\begin{equation*}
	e = e_1 q_1 = e_2 q_2 \text{ and } f = f_1 g_1 = f_2 g_2.
\end{equation*}
$L_2/K$ is purely wild implies $L_2/K$ and $K^r/K$ are linearly disjoint. And $(L_1/K, \nu)$ is tame implies $L_1 \subseteq K^r$. From the definition of linear disjointness, it then follows that $L_1/K$ and $L_2/K$ are linearly disjoint, that is, $[L:K] = [L_1: K][L_2:K]$.
\newline [\ref{Kuh vln model}, Theorem 5.10] implies that $L_2 ^r = L_2 \cdot K^r$. Thus, 
\begin{align*}
	K \subseteq L_1 \subseteq K^r \Longrightarrow L_2 \subseteq L \subseteq L_2 ^r .
\end{align*}
Thus $(L/ L_2, \nu)$ is a tame extension. In particular, $d(L/L_2) = 1$.
\newline We recall that any algebraic extension of a henselian field is again henselian. So in our case, all field extensions are finite extensions of henselian valued fields, hence we can apply Lemma \ref{Ostrowski}. Now $(L_2/K, \nu)$ is purely wild implies $L_2 \nu / K \nu$ is a purely inseparable extension. Since $K \nu$ is perfect, we have $L_2 \nu = K \nu$. Thus, $f_2 = 1$, that is $f = g_2$. Also $\nu L_2 / \nu K$ is a $p$-group. We then have, 
\begin{equation*}
	[L_2 : K] = e_2 d(L_2/K) = p^n p^m \text{ where }n,m \in \NN .
\end{equation*}
Also, $(L/ L_2, \nu)$ is tame implies $[L : L_2] = q_2 g_2 = q_2 f$. Thus, 
\begin{equation*}
	[L:L_2] = [L_1 : K] \Longrightarrow q_2 f = e_1 f_1 \Longrightarrow q_2 g_1 = e_1 \Longrightarrow q_2 \text{ divides } e_1.
\end{equation*}
Again, $e = e_1 q_1 = e_2 q_2 = p^n q_2$. $(L_1/K, \nu)$ is tame implies gcd$(e_1, p) = 1$. Thus gcd$(e_1, p^n) = 1$. So, 
\begin{equation*}
	e_1 q_1 = p^n q_2 \Longrightarrow p^n \mid q_1 \Longrightarrow e_1 \text{ divides } q_2.
\end{equation*}
So we obtain, 
\begin{equation*}
	e_1 = q_2, \, q_1 = p^n = e_2, \, f = f_1, \, g_1 = 1.
\end{equation*}
Thus $L \nu = L_1 \nu$. And $q_1 = p^n$ implies $\nu L / \nu L_1$ is a $p$-group. So $(L/L_1, \nu)$ is a purely wild extension. 
\newline Moreover, $e= q_2 e_2 = e_1 e_2$, that is, $[\nu L : \nu K] = [\nu L_1: \nu K][\nu L_2 : \nu K]$. Hence we have the lemma.

\end{proof}

\begin{Lemma}\label{enough to consider tame for unramified}
	Let $(K, \nu)$ be a henselian field such that $K \nu$ is perfect. Let $(L_1/K, \nu)$ be a finite tame extension and $(L_2/K, \nu)$ be a finite extension. Set $L = L_1 \cdot L_2$. Set $L_2 ^{\prime} = L_2 \sect K^r$ and $L^{\prime} = L_1 \cdot L_2 ^{\prime}$. Then $(L/L_2, \nu)$ and $(L^{\prime}/ L_2 ^{\prime} , \nu)$ are finite tame extensions. Moreover $(L/L_2, \nu)$ is unramified if and only if $(L^{\prime} / L_2 ^{\prime}, \nu)$ is unramified. 
\end{Lemma}

\begin{proof}
	From [\ref{Kuh vln model}, Theorem 5.10] we have that $L_2 ^r = L_2 \cdot K^r$ and $(L_2 ^{\prime})^r = L_2 ^{\prime} \cdot K^r = K^r$. So, $L_2 \sect (L_2 ^{\prime})^r = L_2 \sect K^r = L_2 ^{\prime}$. Hence $(L_2/ L_2 ^{\prime} , \nu)$ is a finite purely wild extension. 
	\newline Now, $L_1 \subseteq K^r$ as $(L_1/K, \nu)$ is a tame extension. So, $L^{\prime} = L_1 \cdot L_2 ^{\prime} \subseteq K^r = (L_2 ^{\prime})^r$ implies $(L^{\prime}/ L_2 ^{\prime} , \nu)$ is finite tame extension. 
	\newline Now, $L^{\prime} \cdot L_2 = L_1 \cdot L_2 ^{\prime} \cdot L_2 = L_1 \cdot L_2 = L$. $(L_2 ^{\prime} /K, \nu)$ is a finite extension, hence $(L_2 ^{\prime}, \nu)$ is a henselian valued field such that $L_2 ^{\prime} \nu$ is perfect. So from Lemma \ref{tame wild lemma} we have
	\begin{equation*}
		[\nu L : \nu L_2 ^{\prime}] = [\nu L^{\prime} : \nu L_2 ^{\prime}] [\nu L_2 : \nu L_2 ^{\prime}] .
	\end{equation*}
	 Again $[\nu L : \nu L_2 ^{\prime}] = [\nu L : \nu L_2] [\nu L_2 : \nu L_2 ^{\prime}]$. Thus we obtain, 
	\begin{equation*}
	[\nu L^{\prime} : \nu L_2 ^{\prime}] = [\nu L : \nu L_2].
	\end{equation*}
	Lemma \ref{tame wild lemma} further implies that $(L/ L_2, \nu)$ is a finite tame extension. Thus $(L/L_2, \nu)$ is umramified, that is, $\nu L = \nu L_2$, if and only if $(L^{\prime} / L_2 ^{\prime}, \nu)$ is unramified.
\end{proof}

\begin{Lemma}\label{both tame, unramified}
	Let $(K, \nu)$ be a henselian field. Let $(L_1/ K, \nu)$ and $(L_2/K, \nu)$ be finite tame extensions. Set $L = L_1 \cdot L_2$ and $e_i = [\nu L_i : \nu K]$, $i = 1,2$. For any $\theta \in \mathfrak{O}_L$ let $d_{\theta} = [L_1 \nu (\theta \nu) : L_1 \nu]$. Then, 
	\begin{equation*}
	(L/L_2, \nu) \text{ is unramified if and only if }  e_2 d_{\theta}= e_1 [ L :  L_1] \text{ for some } \theta \in \mathfrak{O}_L.
	\end{equation*}
	In this case, $L \nu = L_1 \nu (\theta \nu)$ and $e_1 \text{ divides } e_2$.
\end{Lemma}

\begin{proof}
	$(L/L_2, \nu)$ is unramified implies $\nu L = \nu L_2$. Now $L = L_1 \cdot L_2 \subseteq K^r$, hence $(L/K, \nu)$ is tame. In particular, $(L/K, \nu)$ and any subextension of $L/K$ is defectless. Using the multiplicative property of ramification indices and Lemma \ref{Ostrowski}, we then obtain
	\begin{equation*}
		[\nu L : \nu K] = e_2 = [\nu L : \nu L_1] e_1 = \frac{[L:L_1]}{[L \nu : L_1 \nu]} e_1 .
	\end{equation*}
	 Since $L \nu / L_1 \nu$ is a finite separable extension, so $L \nu = L_1 \nu (\theta \nu)$ for some $\theta \in \mathfrak{O}_L$. Then, 
	\begin{equation*}
	e_2 d_{\theta} = e_1 [L : L_1].
	\end{equation*}
	\newline Conversely suppose $e_2 d_{\theta}= e_1 [L : L_1]$ for some $\theta \in \mathfrak{O}_L$. Now $d_{\theta}  = [L_1 \nu (\theta \nu) : L_1 \nu] \leq [L \nu : L_1 \nu]$. So, 
	\begin{equation*}
	e_1 [L : L_1]  = e_2 d_{\theta} \leq e_2 [L \nu : L_1 \nu] \Longrightarrow e_1 [\nu L : \nu L_1] \leq e_2 \Longrightarrow [\nu L : \nu K] \leq e_2.
	\end{equation*}
	Thus all the above inequalities are equalities. So, $e_2 = [\nu L : \nu K]$ and hence $(L/L_2, \nu)$ is unramified. Also, $d_{\theta} = [L \nu : L_1 \nu]$, that is $L \nu = L_1 \nu (\theta \nu)$. And, $d_{\theta} \text{ divides } [L : L_1] \Longrightarrow e_1 \text{ divides }  e_2$.
\end{proof}

From Lemma \ref{enough to consider tame for unramified} and Lemma \ref{both tame, unramified} we obtain the following theorem. 

\begin{Theorem} \label{theorem for unramified}
	Let $(K, \nu)$ be a henselian field such that $K \nu$ is perfect. Let $(L_1/K, \nu)$ be a finite tame extension and $(L_2/K, \nu)$ be a finite extension. Set $L = L_1 \cdot L_2$. Set $L_2 ^{\prime} = L_2 \sect K^r$ and $L^{\prime} = L_1 \cdot L_2 ^{\prime}$. Set $e_1 = [\nu L_1 : \nu K]$  and $e_2 ^{\prime} = [\nu L_2 ^{\prime} : \nu K]$. For any $\theta \in \mathfrak{O}_{L^{\prime}}$ let $d_{\theta} = [L_1 \nu (\theta \nu) : L_1 \nu]$. Then, 
	\begin{equation}\label{condition for unramified}
	(L/L_2, \nu) \text{ is unramified if and only if } e_2 ^{\prime} d_{\theta} = e_1 [L^{\prime} : L_1] \text{ for some } \theta \in \mathfrak{O}_{L^{\prime}}.
	\end{equation} 
	In this case, $L^{\prime} \nu = L_1 \nu (\theta \nu)$ and $e_1 \text{ divides } e_2 ^{\prime}$.
\end{Theorem}

\begin{Corollary}\label{Corollary to unramified theorem when e_1 = e_2}
	Let assumptions and notations be as in Theorem \ref{theorem for unramified}. Suppose that $e_1 = e_2 ^{\prime}$. Then, 
	\begin{equation*}
		(L/L_2, \nu) \text{ is unramified if and only if } (L^{\prime} /L_1, \nu) \text{ is unramified.}
	\end{equation*}
\end{Corollary}

\begin{Remark} $\theta$ mentioned in Theorem \ref{theorem for unramified} arises naturally in the case of complete discretely valued fields of rank 1, hence in particular, $\mathfrak{P}$-adic fields. Since it is enough to look within tame extensions, let $K, \, L_1, \, L_2$ and $L$ be as in Lemma \ref{both tame, unramified}. Further, suppose that $(K, \nu)$ is complete and discretely valued of rank 1. Then $ \exists \,  \theta \in \mathfrak{O}_L$ such that $\mathfrak{O}_L = \mathfrak{O}_K [\theta]$ and $L \nu = K \nu (\theta \nu)$ [cf. \ref{Iyanaga}, Theorem $4.8^{\prime}$]. In particular, $L \nu = L_1 \nu (\theta \nu)$, and hence, $d_{\theta} = [L \nu : L_1 \nu]$. If $(L/L_2, \nu)$ is unramified, we observe that $\theta$ satisfies (\ref{condition for unramified}). 
\end{Remark}

\begin{Remark}
	Let our assumptions and notations be as in Theorem \ref{theorem for unramified}. If $(L/L_2, \nu)$ is unramified, then $e_1$ divides $e_2$. So, $[\nu L : \nu K] =$ lcm$(e_1, e_2)$. However the converse is not true, as illustrated by the next example. 
\end{Remark}

\begin{Example}
	Let $k$ be a field such that char $k \neq 2,3$. Let the valued field $(K, \nu)$ be as in Example \ref{Counterexample to abhyankar lemma}. Let $L_1 = K(\sqrt{X})$ and $L_2 = K(\sqrt[3]{Y})$. Let $L = L_1 \cdot L_2$. Then $(L/ L_2, \nu)$ is ramified.
\end{Example}

\begin{proof}
	Let $e_1 = [\nu L_1 : \nu K], \, e_2 = [\nu L_2 : \nu K], \, e = [\nu L : \nu K]$. From Example \ref{Counterexample to abhyankar lemma} we have $\nu K = (\ZZ \dirsum \ZZ)_{lex}$ and $k \subseteq K \nu$. So char $K \nu \neq 2,3$. We have an unique extension of $\nu$ to $\overline{K}$ with $\nu \overline{K} = (\QQ \dirsum \QQ)_{lex}$. Further we have, 
	\begin{equation*}
		[L_1 : K] = e_1 = 2, \, \nu L_1 = \frac{1}{2} \ZZ \dirsum \ZZ, \, L_1 \nu = K \nu .
	\end{equation*}
	Again, $3 \nu (\sqrt[3]{Y}) = \nu (Y) = (0,1) \Longrightarrow \nu (\sqrt[3]{Y}) = (0, \frac{1}{3}) \notin \nu K$. Using the same arguments as in the proof of Example \ref{Counterexample to abhyankar lemma}, we have
	\begin{equation*}
		[L_2 : K] = e_2 = 3, \, \nu L_2 =  \ZZ \dirsum \frac{1}{3} \ZZ, \, L_2 \nu = K \nu .
	\end{equation*}
	Since char $K \nu \neq 2,3$, both $(L_1/K, \nu)$ and $(L_2/K, \nu)$ are tame extensions. Again $\frac{1}{2} \ZZ \dirsum \frac{1}{3} \ZZ \subseteq \nu L$. So, $e \geq [\frac{1}{2} \ZZ \dirsum \frac{1}{3} \ZZ : \ZZ \dirsum \ZZ] = 6$. Again, $[L : K] \leq [L_1 : K][L_2:K] = 6$. From Lemma \ref{Ostrowski} we then conclude, 
	\begin{equation*}
		[L:K] = e = 6, \, \nu L = \frac{1}{2} \ZZ \dirsum \frac{1}{3} \ZZ, \, L \nu = K \nu . 
	\end{equation*}
	So in particular, $e = $ lcm$(e_1, e_2)$. However $e \neq e_2$, hence $(L / L_2, \nu)$ is ramified. 
\end{proof}

\par We now give a couple of examples illustrating the above results. In the first example, we choose a particular $\theta$ that satisfies the condition (\ref{condition for unramified}). In the second example, we construct a family of tame extensions of henselian valued fields, with different residue characteristics. For each of those tame extensions, we eliminate the ramification after lifting through a similar finite extension. 

\begin{Example}
	Let $k = \mathbb{F}_3 (t)$ where $t$ is an indeterminate and $\nu = \nu_t$ be the $t$-adic valuation on $k$. Let $K = (k, \nu)^h$ be the henselization of $k$ with respect to $\nu$. Let $L_1 = K(\sqrt{t})$ and $L_2 = K(\sqrt[4]{2t})$. Let $L = L_1 \cdot L_2$. Then $(L/L_2, \nu)$ is unramified.  
\end{Example}

\begin{proof}
	Set $e_1 = [\nu L_1 : \nu K]$ and $e_2 = [\nu L_2 : \nu K]$. We know $\nu k = \ZZ$ and $k \nu = \FF_{3}$. Since henselization is an immediate extension, we have $\nu K = \ZZ$ and $K \nu = \FF_3$. The valuation $\nu$ extends uniquely to $\overline{K}$ with $\nu \overline{K} = \QQ$. Now $2 \nu (\sqrt{t}) =  \nu (t) = 1 \Longrightarrow \nu (\sqrt{ t}) = \frac{1}{2}$. Thus $\sqrt{t} \notin K$. Hence
	\begin{equation*}
	[L_1 : K]  = 2 = e_1, \, \nu L_1 = \frac{1}{2} \ZZ, \, L_1 \nu = \FF_3.
	\end{equation*}
	Similarly, observing $\nu(\sqrt[4]{2t}) = \frac{1}{4}$, we obtain
	\begin{equation*}
	[L_2 : K]  = 4 = e_2, \, \nu L_2 = \frac{1}{4} \ZZ, \, L_2 \nu = \FF_3.
	\end{equation*}
	In particular, $(L_1/K, \nu)$ and $(L_2/K, \nu)$ are both tame extensions. Now $L = K(\sqrt{t}, \sqrt[4]{2t}) = L_1 (\sqrt[4]{2t})$. Thus, 
	\begin{equation*}
	[L : L_1] = [L_1 (\sqrt[4]{2t}) : L_1] \leq 4.
	\end{equation*}
	Again $\nu (\sqrt[4]{2t}) = \frac{1}{4} \Longrightarrow \frac{1}{4} \ZZ \subset \nu L$. Since $\nu L_1 = \frac{1}{2} \ZZ$ we have
	\begin{equation*}
	[\nu L : \nu L_1] \geq 2.
	\end{equation*}
	We also observe that $\frac{(\sqrt[4]{2t})^2}{\sqrt{t}} = \sqrt{2} \in L$. $2 \nu(\sqrt{2}) = \nu(2) = 0 \Longrightarrow \nu (\sqrt{2}) = 0$. Thus $\sqrt{2} \in \mathfrak{O}_L$. 
	\newline Let $f(X) = X^2 - 2 \in L_1 [X]$. Since $L_1 \nu = \FF_3$ and $\nu $ is trivial on $\FF_3$, we have $(f \nu)(X) = X^2 - 2 \in \FF_3 [X]$. Now $f(\sqrt{2}) = 0 \Longrightarrow \nu (f(\sqrt{2})) > 0 \Longrightarrow (f \nu)(\sqrt{2} \nu) = 0$. Since $(f \nu) (X)$ is irreducible on $\FF_3$, it is the minimal polynomial of $\sqrt{2} \nu$ over $\FF_3$. Thus $[L_1 \nu (\sqrt{2} \nu) : L_1 \nu ] = [\FF_3 (\sqrt{2} \nu) : \FF_3] = 2$. So, 
	\begin{equation*}
	[L \nu : L_1 \nu] \geq [L_1 \nu (\sqrt{2} \nu) : L_1 \nu]  = 2.
	\end{equation*}
	Since $(K, \nu)$ is henselian and $(L/L_1, \nu)$ is a tame extension, by Lemma \ref{Ostrowski} we have that $[L : L_1] = [\nu L : \nu L_1] [L \nu : L_1  \nu]$. So from the preceding discussion, it follows
	\begin{equation*}
	[L : L_1] = 4, \, [\nu L : \nu L_1] = 2, \, [L \nu : L_1 \nu] = 2.
	\end{equation*}
	Denoting $d_{\sqrt{2}} = [L_1 \nu (\sqrt{2} \nu) : L_1 \nu]$, we then have $e_1 [L : L_1] =  e_2 d_{\sqrt{2}}$. Hence by Lemma \ref{both tame, unramified} we conclude that $(L/L_2, \nu)$ is unramified. Also we obtain $L \nu = \FF_3 (\sqrt{2} \nu)$.
	
\end{proof}

\begin{Example}
Let $p$ be an odd prime such that $p \equiv 3 (\text{mod } 4)$. Let $k = \mathbb{F}_p (t)$ where $t$ is an indeterminate and $\nu = \nu_t$ be the $t$-adic valuation on $k$. Let $K = (k, \nu)^h$ be the henselization of $k$ with respect to $\nu$. Let $L_1 = K(\sqrt{t})$. Then there exists a finite extension $(L_2/K, \nu)$ such that $(L_1 \cdot L_2 / L_2, \nu)$ is uniformizable.   
\end{Example}

\begin{proof}
	We have $\nu K = \ZZ$ and $K \nu = \FF_p$. Denoting $e_1 = [\nu L_1: \nu K]$, from the earlier example we have that $\nu (\sqrt{t}) = \frac{1}{2}$. Further, 
	\begin{equation*}
		[L_1 : K] = e_1 = 2, \, \nu L_1 = \frac{1}{2} \ZZ , \, L_1 \nu = \FF_p.
	\end{equation*}
	Let $\theta = \sqrt{-1} \in \overline{\FF_p}$. Then $2 \nu (\theta) = \nu (-1) = 0 \Longrightarrow \nu (\theta) = 0$. Let $L_2 = K(\theta \sqrt{t}) = K(\sqrt{-t})$. Thus $[L_2 : K ] \leq 2$. Again as before $\nu (\sqrt{-t}) = \frac{1}{2}$. Denoting $e_2 = [\nu L_2 : \nu K]$, we then have
	\begin{equation*}
		[L_2 : K] = e_2 = 2, \, \nu L_2 = \frac{1}{2} \ZZ , \, L_2 \nu = \FF_p.
	\end{equation*} 
	Let $L = L_1 \cdot L_2 = L_1 (\theta)$. Then $[L : L_1] \leq 2$. Now let $f(X) = X^2 + 1 \in L_1 [X]$. Since $\nu$ is trivial on $\FF_p$ we thus have $(f \nu)(X) = X^2 + 1 \in \FF_p [X]$. Then $f(\theta) = 0 \Longrightarrow (f \nu)(\theta \nu) = 0$. Again, $p \equiv 3 (\text{mod}4)$ implies $(f \nu)(X) \in \FF_p [X]$ is irreducible. Thus, $d_{\theta} = [\FF_p (\theta \nu) : \FF_p] = 2$. So $[L \nu : L_1 \nu] \geq 2$. Thus, 
	\begin{equation*}
		[L : L_1] = [L \nu : L_1 \nu], \, \nu L = \nu L_1.
	\end{equation*}
	Since $(L_1/K, \nu)$ and $(L_2/K, \nu)$ are tame extensions, we have that $(L/L_1, \nu)$ is tame. Hence $(L/L_1, \nu)$ is unramified. Again $e_1 = e_2$. From Corollary \ref{Corollary to unramified theorem when e_1 = e_2} it follows that $(L/L_2, \nu)$ is unramified. In particular $L \subseteq L_2 ^i$. 
	\newline Now $(K, \nu)$ is henselian and $(L_2/K, \nu)$ is algebraic implies that $(L_2, \nu)$ is henselian. Also $(L/L_2, \nu)$ is a finite extension with $L \subseteq L_2 ^i$. Then from [\ref{Kuh vln model}, Proposition 15.8], it follows that $(L/L_2, \nu)$ is uniformizable. 
\end{proof}

\end{document}